\documentclass[12pt]{article}    
    
\usepackage{latexsym,amsfonts,amsmath,amssymb,epsfig,tabularx,amsthm,dsfont}    
  
\usepackage{tikz}
    
\theoremstyle{definition}     
    
\newtheorem{Def}{Definition}
\newtheorem{Rem}[Def]{Remark}

\newtheorem{Theo}[Def]{Theorem}
\newtheorem{Cor}[Def]{Corollary}

\newcommand\dint{{\, \rm d}}

\newlength{\fixboxwidth}    
\title{Equivalence of anchored and ANOVA spaces via interpolation} 
 
\author{Aicke Hinrichs \\
Institut f\"ur Analysis, Johannes Kepler Universit\"at Linz\\   
Altenberger Str. 69, 4040 Linz, Austria\\   
email: aicke.hinrichs@jku.at \\ 
\qquad
\\
Jan Schneider \\
Institut f\"ur Mathematik, Friedrich Schiller Universit\"at Jena\\   
Ernst-Abbe-Platz 2, 07743 Jena, Germany\\   
email: jan.schneider@uni-jena.de  \\ 
}   
  
\begin{document}    
 
\maketitle    
   
\begin{abstract}    
We consider weighted anchored and ANOVA spaces of functions with first order mixed derivatives bounded in $L_p$.
Recently,  Hefter, Ritter and Wasilkowski established conditions on the weights in the cases $p=1$ and $p=\infty$ 
which ensure equivalence of the corresponding norms uniformly in the dimension or only polynomially dependent on
the dimension. We extend these results to the whole range of $p\in [1,\infty]$.
It is shown how this can be achieved via interpolation. 
\end{abstract}     

\bigskip 
 
 \medskip 
 
 
\section{Introduction}    

In this paper we continue the study of norm equivalences of anchored and ANOVA spaces of functions with bounded order one mixed partial derivatives. Equivalence of these norms in the Hilbert space case was recently shown by M. Hefter and K. Ritter in \cite{HR14}. They prove very general norm equivalences for reproducing kernel Hilbert spaces. The main point is the possibility of transferring error bounds for algorithms and complexity bounds from one setting to another. 

One example is the multivariate decomposition method or changing dimension algorithm introduced in \cite{KSWW10}. This method has been shown to be very effective for the
anchored decomposition  for spaces of functions with mixed partial derivatives bounded in $L_p$ norms in \cite{W14a}. To transfer these results to the ANOVA setting, a norm equivalence is needed. In the case $p=2$ this follows from the results in \cite{HR14}. In the cases $p=1$ and $p=\infty$, such norm equivalences where recently established by 
M. Hefter, K. Ritter and G. W. Wasilkowski in  \cite{HRW14}. Naturally, the question arises what happens for $1<p<2$ and $2<p<\infty$.  
It is the purpose of this note to derive the corresponding results via interpolation methods.

In Section \ref{sec2}, we introduce the considered spaces of functions with mixed derivatives and state our main general result.
In Section \ref{sec3}, we characterize the spaces as $\ell_p$-sums of certain $L_p$-spaces, which is a suitable characterization for interpolation.
In Section \ref{sec4}, we proof the main result.
In Section \ref{sec5}, we apply the result to different classes of special weights.

\section{Spaces of multivariate functions} \label{sec2}

In this section we introduce the ANOVA and anchored spaces as completions of algebraic tensor products. 
This approach is slightly different from the one taken in \cite{HRW14} but the resulting spaces are the same.
Our approach is better suited to the direct application of complex interpolation, which is the main technical tool later.

Let $1\leq p\leq\infty$.
In the case $d=1$ we consider the space $F_p$ of  complex valued absolutely continuous functions on the interval $[0,1]$ with first derivatives bounded in $L_p$.
We work with complex valued functions since the direct application of the complex interpolation method needs complex scalars.
The following considerations can be done in both the real and complex case. For complex scalars, we just need to consider 
a complex valued function $f:[0,1]^d \to \mathbb{C}$ as a sum $f=g+ {\rm i} h$ of real and imaginary part and apply 
derivatives and integrals to both parts.
Now we consider the algebraic tensor product $$ F_{d,p} = \bigotimes_{j=1}^d F_p $$ of functions on $[0,1]^d$.
These functions have mixed derivatives of first order in $L_p$.
We denote by
$$ f^{(u)} = \prod_{j\in u} \frac{\partial }{\partial x_j} f $$
the mixed derivative of $f$ with respect to the variables in the subset $u \subset [d]=\{1,\dots,d\}$.

To denote function values of $d$-variate functions emphasizing different subsets $u \subset [d]$ of coordinates we use the convention to write
$$ f(x_u;t_{u^c}) := f(x_u;t) := f(y) \qquad\text{with } y_j=x_j \text{ for } j\in u  \text{ and } y_j=t_j \text{ for } j\not\in u.$$

Given a sequence $(\gamma_u)_{u\subset [d]}$ of nonnegative weights, we consider two norms on $F_{d,p}$, the anchored norm
\begin{equation}\label{anchnorm} \|f\|_{\pitchfork,d,p} = \left( \sum_{u\subset[d]} \gamma_u^{-p}  \left\| f^{(u)} (\cdot_u;0) \right\|^p_p  \right)^{1/p} = \left( \sum_{u\subset[d]} \gamma_u^{-p}  \int_{[0,1]^u}  \left| f^{(u)} (x_u;0) \right|^p \dint x_u  \right)^{1/p} \end{equation} 
and the ANOVA norm
\begin{eqnarray}\label{Anorm} \|f\|_{A,d,p} &=& \left( \sum_{u\subset[d]} \gamma_u^{-p}  \left\| \int_{[0,1]^{[d] \setminus u}} f^{(u)} (\cdot_u;t) \dint t \right\|^p_p \right)^{1/p} \nonumber\\
&=& \left( \sum_{u\subset[d]} \gamma_u^{-p}  \int_{[0,1]^u} \left| \int_{[0,1]^{[d] \setminus u}} f^{(u)} (x_u;t) \dint t \right|^p \dint x_u \right)^{1/p}.\end{eqnarray}
Here and in the following we do not explicitly state the standard modification required in the case $p=\infty$. 
Moreover, if the weight $\gamma_u$ is zero, we consider only the subspace $F_{d,p}$ of functions $f$ for which the corresponding term in the norm vanishes.

The space $F_{d,p}$ with respect to these norms will be denoted by $G_{\pitchfork,d,p}$ and  $G_{A,d,p}$ and its completions by $F_{\pitchfork,d,p}$ and  $F_{A,d,p}$, respectively. 
It was shown in \cite[Proposition 13]{HRW14} that the spaces $F_{\pitchfork,d,p}$ and  $F_{A,d,p}$ coincide independently of $p\in[1,\infty]$ if and only if the weights satisfy the
compatibility condition
$$ \gamma_u > 0 \qquad \text{implies} \qquad \gamma_v > 0 \text{ for all } v\subset u.$$  
From this point on we always assume that this condition is fulfilled.
Moreover, it was also observed in \cite{HRW14} that $F_{\pitchfork,d,p}$ and  $F_{A,d,p}$ can be identified with Banach function spaces with continuous point evaluations.
The closed graph theorem then ensures that the identity mappings 
$$J_{d,p}^{\pitchfork,A} :  F_{\pitchfork,d,p} \to  F_{A,d,p} \qquad \mbox{and} \qquad J_{d,p}^{A,\pitchfork} :  F_{A,d,p} \to  F_{\pitchfork,d,p}$$
are bounded. Our main interest in this note is to estimate the corresponding operator norms.

The cases $p=1$ and $p=\infty$ were considered in \cite{HRW14}. 
There the following constants were introduced
\begin{equation}\label{const}C_{d,1}=\max_{u\subset[d]}\sum_{v\subset u}\frac{\gamma_u}{\gamma_v}\quad\text{ and }\quad C_{d,\infty}=\max_{u\subset[d]}\sum_{v\subset u^c} 2^{-|v|}\frac{\gamma_{u\cup v}}{\gamma_u}.\end{equation}
The main result in \cite{HRW14} formulated as Theorem 14 there reads as
\begin{Theo}\label{HRW}
 Let $p=1$ or $p=\infty$. Then 
 $$ \| J_{d,p}^{\pitchfork,A} :  F_{\pitchfork,d,p} \to  F_{A,d,p} \| = \| J_{d,p}^{A,\pitchfork} :  F_{A,d,p} \to  F_{\pitchfork,d,p} \| = C_{d,p}.$$
\end{Theo}
Here we should observe that in \cite{HRW14} real valued functions are treated. But the lower bounds for the norms proved in \cite{HRW14} of course also hold for
complex valued functions. Moreover, the proofs of the upper bounds remain valid also in the complex case. 
This is due to the fact that the inequalities used are triangle inequalities also valid for complex scalars.

Our main result is the following interpolation theorem which extends the upper bounds on the operator norm to $1<p<\infty$. 

\begin{Theo}\label{mainthm}
 Let $1\le p \le \infty$. Then
\begin{equation}\label{mainest} 
    \| J_{d,p}^{\pitchfork,A} :  F_{\pitchfork,d,p} \to  F_{A,d,p} \| \le C_{d,1}^{1/p} C_{d,\infty}^{1-1/p}
    \quad \mbox{and} \quad
		\| J_{d,p}^{A,\pitchfork} :  F_{A,d,p} \to  F_{\pitchfork,d,p} \| \le C_{d,1}^{1/p} C_{d,\infty}^{1-1/p}.
\end{equation}
\end{Theo}

We now recall the notions of uniform and polynomial equivalence from \cite[Definition 14]{HRW14}.
The spaces $F_{\pitchfork,d,p}$ and $F_{A,d,p}$ are called uniformly equivalent (in $d$), if there is 
$c>0$ not depending on $d$, such that
$$ \max\left(\| J_{d,p}^{\pitchfork,A}\|,\| J_{d,p}^{A,\pitchfork}\|\right) \le c,$$
which means that
$$ c^{-1}\,\|f\|_{\pitchfork,d,p}\leq\|f\|_{A,d,p}\leq c\,\|f\|_{\pitchfork,d,p} \text{  for all } f \in F_{\pitchfork,d,p}.$$
The spaces $F_{\pitchfork,d,p}$ and $F_{A,d,p}$ are called polynomially equivalent (in $d$), if there is $\tau>0$ not depending on $d$, such that
$$ \max\left(\| J_{d,p}^{\pitchfork,A}\|,\| J_{d,p}^{A,\pitchfork}\|\right)=O(d^\tau),$$
which means that 
$$ c^{-1}d^{-\tau}\|f\|_{\pitchfork,d,p}\leq\|f\|_{A,d,p}\leq cd^\tau\|f\|_{\pitchfork,d,p}  \text{  for all } f \in F_{\pitchfork,d,p}$$
with a constant $c>0$ independent of $d$.
The infimum over all such $\tau$ is called the exponent of polynomial equivalence.
Of course, these concepts can be formulated more general for sequences of spaces indexed by the dimension, where each element is equipped with two norms. 

As a corollary of Theorem \ref{mainthm}, we obtain immediately from \cite[Proposition 17]{HRW14} the sufficiency in the following corollary. The necessity will be shown in
Section \ref{sec5}.
\begin{Cor}\label{maincor} 
 Let $1\le p\le\infty$.
 For product weights $ \gamma_u = \prod_{j\in u} \gamma_j$, a necessary and sufficient condition for the uniform equivalence of $F_{\pitchfork,d,p}$ and $F_{A,d,p}$ is that $(\gamma_j)_{j\in\mathbb{N}}$ is summable.
\end{Cor}

Further applications to different classes of weights will be considered in Section \ref{sec5}.

\section{Characterization of anchored and ANOVA spaces as $\ell_p(A_j)$-sums}\label{sec3}

Given normed spaces $(A_u)_{u\subset [d]}$, a weight sequence $(\gamma_u)_{u\subset [d]}$ and $p\in[1,\infty]$, the weighted $\ell_p$-sum $\ell_p(A_u)$ of the spaces $A_u$ is the set of
$2^d$-tuples $a=(a_u)_{u\subset [d]}$ with $a_u\in A_u$ and norm
$$ \|a\| = \left( \sum_{u\subset [d]} \gamma_u^{-p} \|a_u\|^p \right)^{1/p}.$$
If $\gamma_u=0$, the corresponding $A_u$ has to be the trivial normed space.

In this section we explain how  the spaces $F_{\pitchfork,d,p}$ and $F_{A,d,p}$ are characterized isometrically as spaces $\ell_p\left(L_p([0,1]^u)\right)$.
Since the completion of $\ell_p(A_u)$ is $\ell_p(\overline{A_u})$, where $\overline{A_u}$ is the completion of $A_u$ we first work in the algebraic tensor products $ F_{d,p} = \bigotimes_{j=1}^d F_p $
equipped with the corresponding norms, i.e.  in $G_{\pitchfork,d,p}$ and  $G_{A,d,p}$. For each function $f\in F_{d,p}$ we consider the operators
$$ R_\pitchfork f = \big( f^{(u)} (x_u;0) \big)_{u\subset[d]}  \qquad \mbox{and} \qquad R_A f = \left( \int_{[0,1]^{[d] \setminus u}} f^{(u)} (\cdot_u;t) \dint t \right)_{u\subset[d]}$$
mapping $f$ to $2^d$-tuples $(g_u^\pitchfork)$ and  $(g_u^A)$, respectively, where the functions $g_u$ depend only on the coordinates in $u$. Note that $g_\emptyset^\pitchfork=f(0)$ and $g_\emptyset^A=\int_{[0,1]^{[d]}} f(t) \dint t$. Hence the norm of $f$ in \eqref{anchnorm} and \eqref{Anorm} can be written as
\begin{equation}\label{iso1} \|f\|_{\pitchfork,d,p} = \left( \sum_{u\subset [d]} \gamma_u^{-p} \|g_u^\pitchfork\|^p_p \right)^{1/p} \qquad \mbox{and} \qquad \|f\|_{A,d,p} = \left( \sum_{u\subset[d]} \gamma_u^{-p} \|g_u^A\|^p_p \right)^{1/p} \end{equation}
where the norms $\|g_u^\pitchfork\|_p$ and $\|g_u^A\|_p$ are $L_p$-norms on the domain $[0,1]^u$. 

We now observe that  $R_\pitchfork f$ and $R_A f$ are actually elements of the spaces $\ell_p(B_u)$ where $B_u \subset A_u = L_p([0,1]^u)$ is the algebraic tensor product $ B_u = \bigotimes_{j\in u} L_p[0,1] $ (with $B_\emptyset=A_\emptyset=\mathbb{R}$).
It also follows from \eqref{iso1} that the mappings $R_\pitchfork$ and $R_A$ considered as mappings from  $ F_{d,p} = \bigotimes_{j=1}^d F_p $
to $\ell_p (B_u)$ are isometric embeddings, i.e.
\[\|f\|_{\pitchfork,d,p} = \|R_\pitchfork f|l_p(B_u)\|\quad\text{and}\quad\|f\|_{A,d,p} = \|R_A f|l_p(B_u)\|.\]

We now show that they actually are isometric isomorphisms. 
In the anchored case, given $g=(g_u) \in \ell_p(B_u)$, we let
$$ f(x) = (S_\pitchfork g) ( x ) = \sum_{u\subset [d]} \int_{[0,x]^u} g_u(t;0) \dint t = \sum_{u\subset [d]} f_u(x) .$$
If $g_u(x)=\prod_{j\in u} h_j(x_j)$ is an elementary tensor in $ B_u = \bigotimes_{j\in u} L_p[0,1] $, then
$$ f_u (x) = \int_{[0,x]^u} g_u(t;0) \dint t = \prod_{j\in u} \int_0^{x_j} h_j(t_j) \dint t_j $$
shows that $$ f_u \in  \bigotimes_{j\in u} F_p  \subset \bigotimes_{j=1}^d F_p = F_{d,p}. $$
Consequently, $S_\pitchfork$ indeed maps $\ell_p(B_u)$ into $F_{d,p}$.

Now $f_u$ depends only on the coordinates in $u$, hence $f_u^{(v)}=0$ if $v  \not\subset u$.
Moreover, if $v \subset u$ and $v\neq u$, then $f_u^{(v)} ( \cdot_v ; 0) = 0$, since for the coordinates in $u \setminus v$ the integration in
$$ f_u(x) = \int_{[0,x]^u} g_u(t;0) \dint t$$
is over the trivial interval $[0,0]$.
Hence $f_u^{(v)} ( \cdot_v ; 0) = 0$ unless $u=v$. Consequently,
$$ f^{(u)} ( \cdot_u ; 0) = f_u^{(u)} ( \cdot_u ; 0)=g_u,$$
which shows that $\|S_\pitchfork g\|_{\pitchfork,d,p}=\|g|l_p(B_u)\|$. Therefore $S_\pitchfork$ is also an isometric embedding. Using the above notation together with (1) in
\cite[Lemma 1]{HRW14}, we find that
\[S_\pitchfork[R_\pitchfork f](x)=S_\pitchfork[(g_u^\pitchfork)_u](x)=\sum_{u\subset [d]} \int_{[0,x]^u} f^{(u)}(t;0) \dint t=f(x).\]
Observe that Lemma 1 in \cite{HRW14} is formulated only for $f$ in the algebraic tensor product of $C^\infty$-functions, but the proof immediately extends to $f \in F_{d,p}$.
This shows that $S_\pitchfork R_\pitchfork = Id_{ F_{d,p}}$.
Since injectivity of $S_\pitchfork$ together with $S_\pitchfork R_\pitchfork = Id_{ F_{d,p}}$ automatically implies $R_\pitchfork S_\pitchfork = Id_{\ell_p(B_u)}$,
we have shown that $S_\pitchfork$ and $R_\pitchfork$ are isometric isomorphisms.

In the ANOVA case we use the one-dimensional identity
$$ f(x) = \int_0^1 f(t) \dint t + \int_0^1 \left( t - \chi_{[x,1]}(t) \right) f'(t) \dint t $$ 
for $f\in F_p$ where
$\chi_{[a,b]}$ is the indicator function of the interval $[a,b]$.
Taking tensor products, this leads to the $d$-dimensional identity
\begin{equation}\label{anovid}
 f(x) = \sum_{u\subset [d]}  \int_{[0,x]^u} \prod_{j\in u} \left( t_j - \chi_{[x_j,1]}(t_j) \right) g_u(t) \dint t 
\end{equation} 
for $f \in  F_{d,p}$
where 
$$ g_u = \int_{[0,1]^{[d] \setminus u}} f^{(u)} (\cdot_u;t) \dint t $$
is the corresponding term in the ANOVA-decomposition of $f$, compare (2) in \cite[Lemma 1]{HRW14}.
So, given $g=(g_u) \in \ell_p(B_u)$, we let
$$ f(x) = (S_A g) ( x ) = \sum_{u\subset [d]} \int_{[0,x]^u} \prod_{j\in u} \left( t_j - \chi_{[x_j,1]}(t_j) \right) g_u(t) \dint t = \sum_{u\subset [d]} f_u(x). $$ 
If $g_u(x)=\prod_{j\in u} h_j(x_j)$ is an elementary tensor in $ B_u = \bigotimes_{j\in u} L_p[0,1] $, then
$$ f_u (x) = \int_{[0,x]^u} \prod_{j\in u} \left( t_j - \chi_{[x_j,1]}(t_j) \right) g_u(t) \dint t = \prod_{j\in u} \int_0^{x_j} \left( t_j - \chi_{[x_j,1]}(t_j) \right) h_j(t_j) \dint t_j $$
shows that $$ f_u \in  \bigotimes_{j\in u} F_p  \subset \bigotimes_{j=1}^d F_p = F_{d,p}. $$
Consequently, $S_A$ maps $\ell_p(B_u)$ into $F_{d,p}$.

Now $f_u$ depends only on the coordinates in $u$, hence $f_u^{(v)}=0$ if $v  \not\subset u$.
Moreover, if $v \subset u$ and $v\neq u$, then 
$$ \int_{[0,1]^{[d] \setminus v}} f_u^{(v)} (\cdot_v;t) \dint t  = 0$$
since the variables $x_j$ with $j \in u \setminus v$ lead to integrals $\int_0^1 \left( t_j - \chi_{[x_j,1]}(t_j) \right) \dint x_j = 0$.
Consequently,
$$ \int_{[0,1]^{[d] \setminus u}} f^{(u)} (\cdot_u;t) \dint t = g_u,$$
which shows that $\|S_A g\|_{A,d,p}=\|g|l_p(B_u)\|$. Therefore $S_A$ is also an isometric embedding. 
Using the above notation together with \eqref{anovid}, we find that
\[S_A[R_A f](x)=S_A[(g_u^A)_u](x)= \sum_{u\subset [d]}  \int_{[0,x]^u} \prod_{j\in u} \left( t_j - \chi_{[x_j,1]}(t_j) \right) g_u^A(t) \dint t  = f(x).\]
This shows that $S_A R_A = Id_{ F_{d,p}}$.
Since injectivity of $S_A$ together with $S_A R_A = Id_{ F_{d,p}}$ automatically implies $R_A S_A = Id_{\ell_p(B_u)}$,
we have shown that $S_A$ and $R_A$ are also isometric isomorphisms.

Now extending  $S_\pitchfork$, $S_A$, $R_\pitchfork$, $R_A$ to the completions and observing that the completion of $B_u$ is just $L_p([0,1]^u)$, 
we obtain isometric isomorphisms of the spaces $F_{\pitchfork,d,p}$ and $F_{A,d,p}$ on the one side and $\ell_p\left(L_p([0,1]^u)\right)$ on the other side.
Slightly abusing notation, we will denote these extensions again with $S_\pitchfork$, $S_A$, $R_\pitchfork$, $R_A$, respectively.

\begin{Rem}
 If the weights $\gamma_u$ are all positive, then we obtain that the spaces $F_{\pitchfork,d,p}$ and $F_{A,d,p}$ are actually the classical spaces $W_p^{mix}$ of functions in $f \in L_p$ whose
 mixed derivative $\frac{\partial^d f}{\partial x_1 \dots \partial x_d}$ also belongs to $L_p$, 
 see \cite{HRW14}. If some of the weights $\gamma_u$ are zero, we obtain subspaces of  $W_p^{mix}$ where the corresponding
 derivatives $f^{(u)}$ satisfy 
 $$ f^{(u)} (x_u;0) =0 \qquad \mbox{and} \qquad \int_{[0,1]^{[d] \setminus u}} f^{(u)} (\cdot_u;t) \dint t=0$$ 
 in the anchored case or in the ANOVA case, respectively.
\end{Rem}

\section{Proof of Theorem \ref{mainthm}}\label{sec4}

In the previous section, we have shown that the spaces $F_{\pitchfork,d,p}$ and $F_{A,d,p}$ are 
isometrically isomorhic to $\ell_p\left(L_p([0,1]^u)\right)$. Now we shortly state known results for this type of spaces for complex interpolation. The corresponding theory is described in detail in \cite{Tr78}, Section 1.9. The relevant results for spaces of type $\ell_p(A_u)$ with $A_u=L_p([0,1]^u)$ are Theorem 1.18.1 (formula (4)) in \cite{Tr78}, i.e.
\begin{equation}\label{intpol1}\left[\ell_{p_0}(A_j),\ell_{p_1}(B_j)\right]_\theta=\ell_p\left([A_j,B_j]_\theta\right),\end{equation}
complemented by the following Remark 2, and Theorem 1.18.6.2 (formula (15)) in \cite{Tr78}, i.e.
\begin{equation}\label{intpol2}[L_{p_0}(A),L_{p_1}(A)]_\theta=L_p(A),\end{equation}
for Banach spaces $A,A_j,B_j$ and $1/p=(1-\theta)/p_0+\theta/p_1$ with $0<\theta<1$.
We want to apply \eqref{intpol1} and \eqref{intpol2} to the situation corresponding to Figure \ref{interpol} to prove Theorem \ref{mainthm}. 

\begin{figure}[h]
\centering
\unitlength1cm
\begin{tikzpicture}[thick]
\draw (0.5,0) -- (1.4,0); \draw (2.6,0) -- (3.4,0); \draw (0.5,-3) -- (1.4,-3); \draw (2.6,-3) -- (3.4,-3);
\draw[->] (0,-0.5) -- (0,-2.5); \draw[->] (2,-0.5) -- (2,-2.5); \draw[->] (4,-0.5) -- (4,-2.5);
\draw (0,0) node {$F_{\pitchfork,d,1}$}; \draw (4,0) node {$F_{\pitchfork,d,\infty}$}; \node at (2,0) [rectangle,draw] {$F_{\pitchfork,d,p}$};
\draw (0,-3) node {$F_{A,d,1}$}; \draw (4,-3) node {$F_{A,d,\infty}$}; \node at (2,-3) [rectangle,draw] {$F_{A,d,p}$};
\draw (-0.4,-1.5) node {$J_{d,1}^{\pitchfork,A}$}; \draw (1.6,-1.5) node {$J_{d,p}^{\pitchfork,A}$};
\draw (3.6,-1.5) node {$J_{d,\infty}^{\pitchfork,A}$}; 
\end{tikzpicture}
\caption{\label{interpol}Interpolation scheme}
\end{figure}
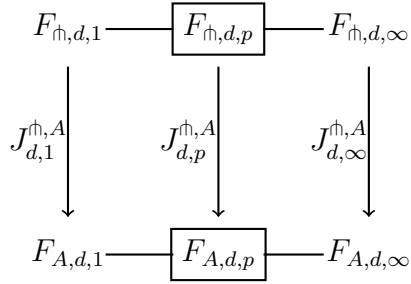

For the operators $J_{d,p}^{\pitchfork,A}$ and $J_{d,p}^{A,\pitchfork}$ (with arrows in Figure \ref{interpol} pointing in the opposite direction), we know from \cite{HRW14}, see 
Theorem \ref{HRW}, that
\[\left\| J_{d,1}^{\pitchfork,A} \right\| = \left\| J_{d,1}^{A,\pitchfork} \right\| = C_{d,1}\quad\text{ and }\quad \left\| J_{d,\infty}^{\pitchfork,A} \right\| = \left\| J_{d,\infty}^{A,\pitchfork} \right\| = C_{d,\infty},\]
with $C_{d,1},C_{d,\infty}$ from \eqref{const}. Now, \eqref{intpol1} and \eqref{intpol2} give
\[\left[\ell_1\left(L_1([0,1]^u)\right),\ell_\infty\left(L_\infty([0,1]^u)\right)\right]_\theta=\ell_p\left(L_p([0,1]^u)\right)\quad\text{ with }\quad p=\frac{1}{1-\theta},\quad 0<\theta<1.\]
Using the isometric isomorphisms $S_\pitchfork$, $S_A$, $R_\pitchfork$, $R_A$, this characterizes the spaces $F_{\pitchfork,d,p}$ and $F_{A,d,p}$ as complex interpolation spaces
$$ F_{\pitchfork,d,p} = \left[ F_{\pitchfork,d,1} , F_{\pitchfork,d,\infty} \right]_\theta \qquad \text{and}\qquad F_{A,d,p} = \left[ F_{A,d,1} , F_{A,d,\infty} \right]_\theta $$
justifying the above diagram.

Because the corresponding interpolation functor is exact of type $\theta$ (see Theorem 1.9.3 in \cite{Tr78}, we get
\[\left\| J_{d,p}^{\pitchfork,A} \right\| \le \left\| J_{d,1}^{\pitchfork,A} \right\|^{1-\theta} \left\| J_{d,\infty}^{\pitchfork,A} \right\|^{\theta} = C_{d,1}^{1/p}C_{d,\infty}^{1-1/p}=:C_{d,p},\]
which is the left side of \eqref{mainest}. Since the arguments are identical for $J_{d,p}^{A,\pitchfork}$, we also get 
\[\left\| J_{d,p}^{A,\pitchfork} \right\| \le \left\| J_{d,1}^{A,\pitchfork} \right\|^{1-\theta} \left\| J_{d,\infty}^{A,\pitchfork} \right\|^{\theta} = C_{d,p},\]
which completes the proof.

\begin{Rem}
 For some potential applications it might be useful to use different weight sequences for different $p$.
 This is possible without much difficulties. Let the spaces $F_{\pitchfork,d,p}$ and $F_{A,d,p}$ be defined in \eqref{anchnorm} and \eqref{Anorm}, respectively,  
 with a weight sequence $(\gamma_{u,p})_{u\subset [d]}$ depending on $p$. Then Theorem 1.8.10.5 in \cite{Tr78} shows that the complex interpolation spaces
 with index $\theta$ between $p=1$ and $p=\infty$ are just given by the interpolated weights $ \gamma_{u,p} = \gamma_{u,1}^{1-\theta} \cdot \gamma_{u,\infty}^\theta$.
 Then Theorem \ref{mainthm} holds verbatim.
\end{Rem}

\section{Applications to different classes of weights}\label{sec5}

We now consider special classes of weights as in \cite{HRW14} and extend the corresponding results from $p=1,\infty$ to the whole range $p\in [1,\infty]$, if
that is possible.

We start with {\em product weights} introduced in \cite{SW98}, for which 
\[\gamma_u=\prod_{j\in u}\gamma_j,\quad\text{ for } u\subset [d]\]
where $(\gamma_j)_{j\in\mathbb{N}}$ is a given sequence of positive numbers. 
Product weights  were already mentioned in Corollary \ref{maincor}. 
To prove that corollary, it remains to show that the summability of $(\gamma_j)_{j\in\mathbb{N}}$ is necessary for the uniform equivalence of $F_{\pitchfork,d,p}$ and $F_{A,d,p}$ in case $p\in[1,\infty]$. 
To this end, we consider the function
\[f(x)=f(x_1,\ldots,x_d)=\prod_{j=1}^d(1+\gamma_jx_j).\]
Then we have
\[f^{(u)}(\cdot_u,0)=\gamma_u\quad\text{ and }\quad \int_{[0,1]^{[d] \setminus u}} f^{(u)} (\cdot_u;t) \dint t=\gamma_u\prod_{j\notin u}(1+\gamma_j/2)\]
and calculate
\[\|f\|_{\pitchfork,d,p}^p=\sum_{u\subset[d]} \gamma_u^{-p} \int_{[0,1]^u} |f^{(u)}(x_u;0)|^p\dint x_u=2^d\]
as well as 
\begin{eqnarray*}\|f\|_{A,d,p}^p&=&\sum_{u\subset[d]} \gamma_u^{-p}  \int_{[0,1]^u} \left| \int_{[0,1]^{[d] \setminus u}} f^{(u)} (x_u;t) \dint t \right|^p \dint x_u=\sum_{u\subset[d]}\prod_{j\notin u}(1+\gamma_j/2)^p\\
&=&\prod_{j=1}^d(1+(1+\gamma_j/2)^p).\end{eqnarray*}
Now we see 
\[\left\| J_{d,p}^{\pitchfork,A} \right\|^p\geq\frac{\|f\|_{A,d,p}^p}{\|f\|_{\pitchfork,d,p}^p}=\prod_{j=1}^d\frac{1+(1+\gamma_j/2)^p}{2}\geq\prod_{j=1}^d (1+\gamma_j/4),\]
which shows that, if the spaces $F_{\pitchfork,d,p}$ and $F_{A,d,p}$ are uniformly equivalent, then $(\gamma_j)_{j\in\mathbb{N}}$ must be summable.
This finishes the proof of Corollary \ref{maincor}.

The corresponding result for polynomial equivalence is
\begin{Cor}
 Let $1\le p\le\infty$.
 For product weights $ \gamma_u = \prod_{j\in u} \gamma_j$, a necessary and sufficient condition for the polynomial equivalence of $F_{\pitchfork,d,p}$ and $F_{A,d,p}$ is that 
 \begin{equation}\label{xxx}
  \tau_0 = \sup_{d\in\mathbb{N}}\frac{\sum_{j=1}^d\gamma_j}{\ln(d+1)}<\infty.
 \end{equation}
 If this holds, then the exponent $\tau$ of polynomial equivalence is at most $$  \frac{\tau_0}{2} \left( 1+\frac1p \right).$$ 
\end{Cor}

\begin{proof} 
For the sufficiency, we now use \cite[Proposition 17, (ii)]{HRW14} showing that polynomial equivalence for $p\in\{1,\infty\}$
holds if and only if $\tau_0<\infty$. Moreover, it is also shown there that the exponents of polynomial equivalence are $\tau_0/2$ for $p=\infty$ and $\tau_0$ for $p=1$, 
so the upper bound for the exponent of polynomial equivalence for $1<p<\infty$ follows from the interpolation inequalities
\[\left\| J_{d,p}^{\pitchfork,A} \right\| \le \left\| J_{d,1}^{\pitchfork,A} \right\|^{1-\theta} \left\| J_{d,\infty}^{\pitchfork,A} \right\|^{\theta} = C_{d,1}^{1/p}C_{d,\infty}^{1-1/p}=C_{d,p}\]
and 
\[\left\| J_{d,p}^{A,\pitchfork} \right\| \le \left\| J_{d,1}^{A,\pitchfork} \right\|^{1-\theta} \left\| J_{d,\infty}^{A,\pitchfork} \right\|^{\theta} = C_{d,1}^{1/p}C_{d,\infty}^{1-1/p}=C_{d,p}.\]
The necessity of \eqref{xxx} follows by the same arguments (for the same function) as used above for the uniform equivalence. 
\end{proof}

Now we discuss the class of so-called {\em finite order weights} first considered in \cite{DSWW} and having the property that
\[\gamma_u=0\quad\text{ if } |u|>q\in\mathbb{N}\quad\text{ ($q$ is called order).}\]
Using \cite[Proposition 19]{HRW14} we can state 
\begin{Cor}
Let $1\le p\le\infty$.
If there exist numbers $c,\omega>0$, such that
\[\gamma_u=c\,\omega^{|u|}\quad\text{ for all $|u|\leq q$},\]
then there is polynomial equivalence of the spaces $F_{\pitchfork,d,p}$ and $F_{A,d,p}$ equipped with finite order weights. The corresponding exponent 
of polynomial equivalence is $\tau=q$.
\end{Cor}

\begin{Rem}
The original condition on the weights $\gamma_u$ in \cite{HRW14} looks a bit different because the cases $p=1$ and $p=\infty$ are treated separately. There one can also see that the only known case where uniform equivalence holds, is $p=1$ for these weights. So here we can not use our interpolation technique to establish the corresponding result for $p\in[1,\infty]$.
\end{Rem}

As a last application we consider here special {\em dimension-dependent weights}, introduced in \cite{HW13}, where $\gamma_u=d^{-|u|}$.  Now \cite[Proposition 20]{HRW14} gives immediately  
\begin{Cor}
For $p\in[1,\infty]$ and weights $\gamma_u=d^{-|u|}$ the spaces $F_{\pitchfork,d,p}$ and $F_{A,d,p}$ are uniformly equivalent.
\end{Cor}

{\bf Acknowledgment:} We thank two anonymous referees for valuable remarks that improved the presentation of our paper.

\end{document}